\documentclass[11pt]{amsart}
\usepackage{amssymb,amscd,  latexsym, graphicx, mathrsfs, enumerate}

\newcommand{\nc}{\newcommand}
\numberwithin{equation}{section}
\newtheorem{theorem}{Theorem}[section]
\newtheorem{prop}[theorem]{Proposition}
\newtheorem{importnota}[theorem]{Important Notation}
\newtheorem{prblm}[theorem]{Problem}
\newtheorem{notation}[theorem]{Notation}
\newtheorem{caution}[theorem]{Caution}
\newtheorem{remark}[theorem]{Remark}
\newtheorem{lemma}[theorem]{Lemma}
\newtheorem{construction}[theorem]{Construction}
\newtheorem{corollary}[theorem]{Corollary}
\newtheorem{example}[theorem]{Example}
\newtheorem{conclusion}[theorem]{Conclusion}
\newtheorem{triviality}[theorem]{Triviality}
\newtheorem{proto}[theorem]{Prototype Quasifibration}
\newtheorem{cauex}[theorem]{Cautionary Example}
\newtheorem{propositiondef}[theorem]{Proposition-Definition}
\newtheorem{subth}{Nuisance}[theorem]
\newtheorem{ssubth}{ }[subth]
\newtheorem{conjecture}[theorem]{Conjecture}
\newtheorem{sidest}[theorem]{Side Story}
\newtheorem{miniexample}[theorem]{Example}
\theoremstyle{definition}
\newtheorem{defin}[theorem]{Definition}

\nc\tri[1]{\begin{triviality}}
\nc\side[1]{\begin{sidest}}
\nc\conj[1]{\begin{conjecture}}
\nc\prodef[1]{\begin{propositiondef}}
\nc\prt[1]{\begin{proto}}
\nc\lem[1]{\begin{lemma}}
\nc\sblm[1]{\begin{sublemma}}
\nc\pro[1]{\begin{prop}}
\nc\thm[1]{\begin{theorem}}
\nc\cor[1]{\begin{corollary}}
\nc\dfn[1]{\begin{defin}}
\nc\sthm[1]{\begin{subth}}
\nc\exm[1]{\begin{example}}
\nc\miniexm[1]{\begin{miniexample}}
\nc\plm[1]{\begin{prblm}}
\nc\rmk[1]{\begin{remark}}
\nc\subrmk[1]{\begin{subremark}}
\nc\ntn[1]{\begin{notation}}
\nc\cau[1]{\begin{caution}}
\nc\imn[1]{\begin{importnota}}
\nc\cax[1]{\begin{cauex}}
\nc\con[1]{\begin{construction}}
\nc\ssthm[1]{\begin{ssubth}}
\nc\cnc[1]{\begin{conclusion}}
\nc\elem{\end{lemma}}
\nc\esblm{\end{sublemma}}
\nc\eside{\end{sidest}}
\nc\econj{\end{conjecture}}
\nc\eprodef{\end{propositiondef}}
\nc\eprt{\end{proto}}
\nc\ethm{\end{theorem}}
\nc\ecor{\end{corollary}}
\nc\edfn{\end{defin}}
\nc\esthm{\end{subth}}
\nc\epro{\end{prop}}
\nc\etri{\end{triviality}}
\nc\eexm{\end{example}}
\nc\eminiexm{\end{miniexample}}
\nc\ermk{\end{remark}}
\nc\subermk{\end{subremark}}
\nc\eplm{\end{prblm}}
\nc\ecau{\end{caution}}
\nc\ecax{\end{cauex}}
\nc\eimn{\end{importnota}}
\nc\entn{\end{notation}}
\nc\econ{\end{construction}}
\nc\ecnc{\end{conclusion}}
\nc\essthm{\end{ssubth}}

\newcommand{\C}{\mathbb{C}}
\newcommand{\R}{\mathbb{R}}
\newcommand{\Q}{\mathbb{Q}}

\newcommand{\Z}{\mathbb{Z}}

\newcommand{\A}{\mathbb{A}}
\newcommand{\E}{\mathcal{E}}
\renewcommand{\O}{\mathcal{O}}

\newcommand{\G}{\Gamma}

\newcommand{\ds}{\displaystyle}

\newcommand{\tr}{{\rm Tr}}

\newcommand{\lra}{\longrightarrow}
\newcommand{\bs}{\backslash}

\newcommand{\D}{\mathfrak{D}}
\renewcommand{\Bbb}{\mathbb}

\title[Ikeda type construction of cusp forms]
{Ikeda type construction of cusp forms}
\author{Henry H. Kim and Takuya Yamauchi}
\keywords{Ikeda type lift, Eisenstein series, Langlands functoriality}
\thanks{The first author is partially supported by NSERC. The second author
is partially supported by JSPS Grant-in-Aid for Scientific Research (C) No.15K04787.}
\subjclass[2010]{}
\address{Henry H. Kim \\
Department of mathematics \\
 University of Toronto \\
Toronto, Ontario M5S 2E4, CANADA \\
and Korea Institute for Advanced Study, Seoul, KOREA}
\email{henrykim@math.toronto.edu}

\address{Takuya Yamauchi \\
Department of mathematics, Faculty of Education\\
Kagoshima University\\
Korimoto 1-20-6 Kagoshima 890-0065, JAPAN}
\email{yamauchi@edu.kagoshima-u.ac.jp}

\begin{document}
\begin{abstract}
This is a survey of results on the construction of holomorphic cusp forms on tube domains
originally initiated by Ikeda \cite{Ik1}. Besides a survey it includes conjectures and possible applications of our work \cite{KY}.
\end{abstract}
\maketitle

\section{Introduction}

There are five simple tube domains (cf. \cite{G}). They are of the form $\mathfrak D=\{Z=X+iY | X\in \Bbb R^n, Y\in C\}$, where $C$ is a self-adjoint homogeneous cone in $\Bbb R^n$. Let $G$ be (the real points of) the simply connected, simple real algebraic group which acts transitively on $\mathfrak D$. We list the group $G$ and the cone $C$:
\begin{enumerate}
\item $Sp_{2n}$ (rank $n$); $n\times n$ positive definite matrices over $\Bbb R$;
\item $SU(n,n)$; $n\times n$ positive definite hermitian matrices over $\Bbb C$;
\item $SU(2n,H)=Spin^*(4n)$; $n\times n$ positive definite hermitian matrices over $H$ (quaternions);
\item $SO(2,n)^0$; the cone in $\Bbb R^{n+1}$ of $(x_0,...,x_n)$ with $x_0> (x_1^2+\cdots+x_n^2)^{\frac 12}$;
\item $E_{7,3}$; $3\times 3$ positive definite hermitian matrices over $\frak C$ (Cayley numbers).
\end{enumerate}

It is an important problem to explicitly construct holomorphic cusp forms on $\D$ with respect to $G(\Z)$
(we will call such a modular form on $\D$ ``a level one form"). In particular, we focus on the lifting from normalized Hecke cusp eigenforms on the complex upper half-plane $\mathbb{H}$ with respect to $SL_2(\Z)$ to holomorphic cusp forms on $\D$.

Ikeda \cite{Ik1} (see also \cite{Ik0}) gave a (functorial) construction of Siegel cusp forms of weight
$n+k,\ n\equiv
k\mod 2$  (so that $n+k$ is even)
for $Sp_{4n}$ from normalized Hecke eigenforms in $S_{2k}(SL_2(\Z))$ which has been conjectured by Duke and Imamoglu
(Independently Ibukiyama formulated a conjecture in terms of Koecher-Maass series). He made use of the uniform property of the Fourier
coefficients of Siegel Eisenstein series for $Sp_{4n}$ and together with various deep facts established in \cite{Ik1}
to prove Duke-Imamoglu conjecture. When $n=1$, it is nothing but a Saito-Kurokawa lift. Since then, his construction was generalized to unitary groups $U(n,n)(K/\Q)$ or $SU(n,n)$
for an imaginary quadratic field $K/\Q$ (\cite{Ik2}),
quaternion unitary groups $SU(2n,H)$ for a definite quaternion algebra $H$ over $\Q$ (\cite{Yamana}), symplectic groups $Sp_{2n}$ over totally real fields (\cite{Ik4},\cite{Ik&H}
including some levels), and the exceptional group of type $E_{7,3}$ with $\Q$-rank 3 \cite{KY}.

In this note we explain main ideas of Ikeda and how they generalize to above cases.
We do not discuss a further development by Ikeda \cite{Ik4} though it is important because his new
ideas will work beyond ``level one" case. We can give a uniform treatment except the case (4), which we will omit since it has been studied thoroughly by Oda \cite{oda} and Sugano \cite{sugano}.

Let $G$ be $Sp_{4n}$, $SU_{2n+1}:=SU(2n+1,2n+1)(K/\Q)$ (to ease the notation, we restrict ourselves to this case),
$SU(2n,H)$, or $E_{7,3}$, and $P=MN$ the Siegel parabolic subgroup of $G$
with the Levi subgroup $M$ and the abelian unipotent radical $N$.
For any ring $R$, let $\tr_G:N(R)\lra R$ be the trace on $N$, which is defined as:
$$\tr_G(n(B)):=\left\{
\begin{array}{ll}
\tr(B) &\ {\rm if}\ G=Sp_{4n}, N=\Bigg\{n(B)=\left(
\begin{array}{cc}
1_{2n} & B \\
0_{2n} & 1_{2n}
\end{array}\right) \ \Bigg|\ {}^tB=B  \Bigg\} \\
\frac{1}{2}\tr(B+\overline{B}) &\ {\rm if}\ G=SU_{2n+1},
N=\Bigg\{n(B):=\left(
\begin{array}{cc}
1_{2n+1} & B \\
0_{2n+1} & 1_{2n+1}
\end{array}\right)\ \Bigg|\ {}^t {\bar B}=B  \Bigg\} \\
\frac{1}{2}\tr(B+\tau(B)) &\ {\rm if}\ G=SU(2n,H),
N=\Bigg\{n(B):=\left(
\begin{array}{cc}
1_{n} & B \\
0_{n} & 1_{n}
\end{array}\right)\ \Bigg|\ {}^t ({}^\iota B)=B  \Bigg\},
\end{array}
\right.
$$
where ${}^\iota x=x_0-ix_1-jx_2-kx_3$ for $x=x_0+ix_1+jx_2+kx_3\in H$, and $\tau(x)=x+{}^\iota x$.
\newline For $E_{7,3}$, see \cite{KY}.

Set $K=\Q$ if $G=Sp_{4n}$ or $E_{7,3}$, and $K=\Bbb H$ if $G=SU(2n,H)$.
Let $\O$ be the ring of integers of $K$ if $G\ne SU(2n,H)$, and
a maximal order of $H$ if $G=SU(2n,H)$. An element $T$ of $N(K)$ is semi-integral if
$\tr_G(TX)\in \Z$ for any $N(\O)$. We denote by $L$ the set of all semi-integral elements in $N(K)$ and
denote by $L^+$ the subset of $L$ consisting of positive definite elements.
Here the positivity has the usual meaning as matrices for $G\not=E_{7,3}$, and see \cite{KY} for $E_{7,3}$.
For instance, if $G=Sp_{4n}$, $L$ consists of matrices $(x_{ij})_{1\le i,j\le 2n}$ so that $x_{ii}\in \Z$ and
$x_{ij}=x_{ji}\in \frac{1}{2}\Z$ for $i\not=j$.

For the integers $k$ and $d$, we denote by $\mathfrak{d}_d$ the discriminant of $\Q\left(\sqrt{(-1)^k d}\right)/\Q$ and
$\chi_d$ the Dirichlet character associated to $\Q\left(\sqrt{(-1)^k d}\right)/\Q$. Let $\mathfrak{f}_d$ be the positive
rational number so that $d=\mathfrak{d}_d\mathfrak{f}^2_d$. Let $L(s,\chi_d)$ be the Dirichlet L-function of $\chi_d$.
For $T\in L^+$, put $D_T=\det(2T)$ (resp. $\gamma(T)=(-D_{K})^n\det(T)$ where $-D_{K}$ stands for the
fundamental discriminant of $K/\Q$) if $G=Sp_{4n}$ (resp. if
$G=SU_{2n+1}$). For $G=SU(2n,H)$, put $D_T=(D_H)^{\frac{n}{2}}{\rm Paf}(T)$
where $D_H$ is the product of rational primes $p$ so that $H\otimes_\Q\Q_p$ is a skew field and
Paf is defined in Section 1 of \cite{Yamana}.
When $G=E_{7,3}$, $\det(T)$ is as in \cite{KY}.  Set
$$\ell(k):=\left\{
\begin{array}{ll}
k+n &\ {\rm if}\ G=Sp_{4n},\\
2k+2n &\ {\rm if}\ G=SU_{2n+1},\\
2k+2n-2 &\ {\rm if}\ G=SU(2n,H),\\
2k+8 &\ {\rm if}\ G=E_{7,3}.
\end{array}
\right.
$$
For each $\gamma\in G(\R)$ and $Z\in \D$, one can associate the
automorphic factor $j(\gamma,Z)\in \C$ so that $j(\gamma,Z)^k$ is used to define modular forms on $\D$ of weight $k$
for any integer $k\ge 0$. For example, if $\gamma=\left(
\begin{array}{cc}
A & B \\
C & D
\end{array}\right)\in Sp_{2n}(\R)$, then $j(\gamma,Z)=\det(CZ+D)$.
Put $\G:=G(\Z)$ and $\G_\infty=\G\cap N(\Q)$.
Let us consider the Siegel Eisenstein series
of weight $\ell(k)$:
$$E_{\ell(k)}(Z)=
\ds\sum_{\gamma\in \Gamma_\infty\bs\G}j(\gamma,Z)^{-\ell(k)}.
$$
Then we have the Fourier expansion

$$\E_{\ell(k)}(Z)=\frac{1}{C(\ell(k))}E_{\ell(k)}(Z)=\sum_{T\in L} A(T)\exp(2\pi\sqrt{-1}\cdot \tr_G(TZ)),
$$
for a constant $C(\ell(k))$, and for $T\in L^+$, $A(T)$ is given as follows:
$$
A(T)=\left\{
\begin{array}{ll}
L(1-k,\chi_{D_T})\mathfrak{f}^{k-\frac{1}{2}}_T\ds\prod_{p|\mathfrak{d}_T}\widetilde{F}_{p}(T;p^{k-\frac{1}{2}})  &\ {\rm if}\ G=Sp_{4n} \\
|\gamma(T)|^{k-\frac{1}{2}}\ds\prod_{p|\gamma(T)}\widetilde{F}_{p}(T;p^{k-\frac{1}{2}})  &\ {\rm if}\ G=SU_{2n+1} \\
{D_T}^{k-\frac{1}{2}}\ds\prod_{p|D_T}\widetilde{f}_{p,T}(p^{k-\frac{1}{2}})&\ {\rm if}\ G=SU(2n,H) \\
\det(T)^{k-\frac{1}{2}}\ds\prod_{p|\det(T)} \widetilde{f}^p_T(p^{k-\frac{1}{2}})&\ {\rm if}\ G=E_{7,3},
\end{array}\right.
$$
where $\widetilde{F}_{p}(T;X), \widetilde{f}_{p,T}(X)$ and $ \widetilde{f}^p_T(X)$ are Laurent polynomials over $\Q$ with one variable $X$
which are depending only on
$T,p$ and both are identically 1 for all but finitely many $p$. Introducing
multi-variables $\{X_p\}_p$ indexed by rational primes $p$, we may consider
$$
A(\{X_p\}_p):=\left\{
\begin{array}{ll}
L(1-k,\chi_{D_T})\mathfrak{f}^{k-\frac{1}{2}}_T\ds\prod_{p|\mathfrak{d}_T}\widetilde{F}_{p}(T;X_p)  &\ {\rm if}\ G=Sp_{4n} \\
|\gamma(T)|^{k-\frac{1}{2}}\ds\prod_{p|\gamma(T)}\widetilde{F}_{p}(T;X_p)  &\ {\rm if}\ G=SU_{2n+1} \\
{D_T}^{k-\frac{1}{2}}\ds\prod_{p|D_T}\widetilde{f}_{p,T}(X_p)&\ {\rm if}\ G=SU(2n,H) \\
\det(T)^{k-\frac{1}{2}}\ds\prod_{p|\det(T)} \widetilde{f}^p_T(X_p)&\ {\rm if}\ G=E_{7,3}.
\end{array}\right.
$$
Then $A(\{X_p\}_p)$ can be regarded as
an element of $\otimes'_p\C[X_p,X^{-1}_p]$.
For each normalized Hecke eigenform $f=\ds\sum_{n=1}^\infty a(n)q^n,\ q=\exp(2\pi \sqrt{-1}\tau),\ \tau\in \mathbb{H}$ in $S_{2k}(SL_2(\Z))$ and each rational prime $p$,
we define the Satake $p$-parameter $\alpha_p$ by $a(p)=p^{k-\frac{1}{2}}(\alpha_p+\alpha^{-1}_p)$.
For such $f$, consider the following formal series on $\D$:
$$F_f(Z):=\sum_{T\in L^+} A_{F_f}(T)\exp(2\pi\sqrt{-1}\tr_G(TZ)),\ Z\in \frak D, \quad
A_{F_f}(T)=A(\{\alpha_p\}_p).
$$
Then
\begin{theorem}\label{main-thm1}
Assume that $H$ is the Hurwitz quaternion when $G=SU(2n,H)$. Then
$F_f$ is a non-zero Hecke eigen cusp form on $\D$ of weight $\ell(k)$ with
respect to $G(\Z)$.
\end{theorem}
Of course, we have to specify what kind of Hecke theory we use for each case. At any late,
the issue is only on the normalization factor of a Hecke action and it does not matter as long as we deal with the adelic form
attached to $F_f$ on $G(\A_\Q)$ because since $G$ is semi-simple, it does not contain the central torus.
By virtue of Theorem \ref{main-thm1}, $F_f$ gives rise to a cuspidal automorphic representation
$\pi_F=\pi_\infty\otimes \otimes'_p\pi_p$
of $G(\A_\Q)$.
Here $\pi_\infty$ is a holomorphic discrete series of $G(\Bbb R)$ of the lowest weight $\ell(k)$, and for each prime $p$, $\pi_p$ is unramified at every finite place (but a few exception when $G=SU(2n,H)$), since $F_f$ is of ``level one". In fact, $\pi_p$ turns out to be a
degenerate principal series $\pi_p\simeq I(s_p)$, where $s_p\in \C$ so that $p^{s_p}=\alpha_p$ and
$$I(s)=
\left\{
\begin{array}{ll}
{\rm Ind}_{P(\Q_p)}^{G(\Q_p)}\: |\nu(g)|^{s}_p  & {\rm if}\ G=Sp_{4n}\\
{\rm Ind}_{P(\Q_p)}^{G(\Q_p)}\: |\nu(g)|^{s}_p  & {\rm if}\ G=SU_{2n+1},\\
{\rm Ind}_{P(\Q_p)}^{G(\Q_p)}\: |\nu(g)|^{s}_p  & {\rm if}\ G=SU(2n,H)\ {\rm and}\ p\nmid D_H,\\
{\rm Ind}_{P(\Q_p)}^{G(\Q_p)}\: |\nu(g)|^{2s}_p & {\rm if}\ G=E_{7,3},
\end{array}\right.
$$
where
$\nu: P(\Q_p)\lra \Q^\times_p$ is defined as follows:
$$\nu(g):=\left\{
\begin{array}{ll}
det(A) &\ {\rm if}\ G=Sp_{4n}, P=\Bigg\{g=\left(
\begin{array}{cc}
A & B \\
0_{2n} & {}^t A^{-1}
\end{array}\right) \ \Bigg|\ {}^tB=B  \Bigg\} \\
|det(A)|^2 &\ {\rm if}\ G=SU_{2n+1},
P=\Bigg\{g=\left(\begin{array}{cc}
A & B \\
0_{2n+1} & {}^t {\bar A}^{-1}
\end{array}\right)\ \Bigg|\ {}^t {\bar B}=B  \Bigg\}
\end{array}
\right.
$$
For $SU(2n,H)$ and $E_{7,3}$, see \cite{Yamana} and \cite{KY}, resp.   The relationship between $I(s)$ and the Eisenstein series is explained in \cite{Ku}:
Let $\Phi(g, s)=\Phi_\infty(g, s)\otimes \otimes_p \Phi_p(g, s)$ be a standard section in $I(s)$ such that $\Phi_\infty(k, s)=\nu(\bold k)^{\ell(k)}$, and
$\Phi_p(g,s)=\Phi_p^0(g,s)$ is the normalized spherical section for all $p$.
 Then one can define the adelic and classical Eisenstein series
$$E(g,s,\Phi)=\sum_{\gamma\in \bold P(\Q)\backslash \bold{G}(\Q)} \, \Phi(\gamma g, s), \quad E_{\ell(k),s}(Z)=det(Y)^{\frac s2}
\ds\sum_{\gamma\in \Gamma_\infty\bs\G}j(\gamma,Z)^{-\ell(k)} |j(\gamma,Z)|^{-s}.
$$
Then we have
$$E(g,s,\Phi)=\left\{
\begin{array}{ll}
\det(Y)^{\ell(k)} E_{\ell(k),s+\frac 12-k}(Z), & {\rm if}\ G\ne E_{7,3},\\
\det(Y)^{\frac {s+9}2+\ell(k)} E_{\ell(k),s+1-2k}(Z),  & {\rm if}\ G=E_{7,3},
\end{array}\right.
$$
Hence the degenerate principal series $I(k-\frac 12)$ corresponds to $E_{\ell(k)}(Z)$ if $G\ne E_{7,3}$, and $I(2k-1)$ corresponds to $E_{\ell(k)}(Z)$ if $G=E_{7,3}$.

In terms of representation theory, Theorem \ref{main-thm1} can be reformulated as follows: Let $\pi_\infty$ be the holomorphic discrete series of $G(\Bbb R)$ of the lowest weight $\ell(k)$, and let $\pi_p$ be the above degenerate principal series which is irreducible. Then we can form an irreducible admissible representation of $G(\Bbb A_{\Bbb Q})$: $\pi=\pi_\infty\otimes \otimes_p' \pi_p$.
Then Theorem \ref{main-thm1} is equivalent to the fact that $\pi$ is a cuspidal automorphic representation of $G(\Bbb A)$. In this formulation, at least for $Sp_{4n}$, Arthur's trace formula \cite{A} may give a more general result as follows: By Adams-Johnson's result on $A$-packets, $\pi_\infty$ belongs to a packet with the local character $(-1)^n$.
Since $\pi$ is unramified at every finite place, by the multiplicity formula, $\pi$ is a cuspidal automorphic representation if and only if the global character $(-1)^n$ is equal to the root number of $L(s,f)$ which is $(-1)^k$. Hence we have the parity condition
$k\equiv n$ (mod 2). We have similar results for $SU_{2n+1}$ and $SU(2n,H)$.
However, the advantage of Theorem \ref{main-thm1} is that one can write down the modular form explicitly.
Let $L(s,\pi_f)=\prod_p ((1-\alpha_p p^{-s})(1-\alpha_p^{-1} p^{-s}))^{-1}$ be the (normalized)
automorphic $L$-function of the cuspidal representation $\pi_f$ attached to $f$. In the case of $SU_{2n+1}$, let $\chi(p)=\left(\frac {-D_K}p\right)$ be the quadratic character attached to $K/\Bbb Q$, and
$L(s,f,\chi)=\prod_p ((1-\alpha_p\chi(p) p^{-s})(1-\alpha_p^{-1}\chi(p) p^{-s}))^{-1}$.
For each local component $\pi_p$, one can associate the local $L$-factor $L(s,\pi_p,St)$ of the standard $L$-function of $\pi_F$.
Set $L(s,\pi_F,St)=\ds\prod_p L(s,\pi_p,St)$:

\begin{theorem}\label{main-thm2}
$$L(s,\pi_F,St)= \left\{
\begin{array}{ll}
\zeta(s)\displaystyle \prod_{i=1}^{2n} L(s+n+\frac 12-i,f)  &\ {\rm if}\ G=Sp_{4n},\\
\displaystyle \prod_{i=1}^{2n+1} L(s+n+1-i,f)L(s+n+1-i,f,\chi) &\ {\rm if}\ G=SU_{2n+1} \\
 \displaystyle \prod_{i=1}^{2n} L(s+n+\frac 12-i,f) &\ {\rm if}\ G=SU(2n,H) \\
L(s,{\rm Sym}^3 \pi_f)L(s,f)^2 \ds\prod_{i=1}^4 L(s\pm i,f)^2 \prod_{i=5}^8 L(s\pm i, f) &\ {\rm if}\ G=E_{7,3},
\end{array}\right.
$$
where $L(s,{\rm Sym}^3 \pi_f)$ is the symmetric cube $L$-function.
\end{theorem}
Notice that $\pi_p$ for $G=E_{7,3}$ is slightly different from other cases (Note $2s_p$ rather than $s_p$) and
the third symmetric power $L$-function appears in the standard L-function. Note also that in the case $G=SU_{2n+1}$, $L(s,f)L(s,f,\chi)=L(s, \pi_K)$, the $L$-function of the base change $\pi_K$ of $\pi_f$ to $K$.

In Section 2, we review the tube domains.
In Section 3, we review the Jacobi group,
Jacobi forms, and a key property of the Fourier-Jacobi expansion of
Siegel Eisenstein series, namely, the Fourier-Jacobi coefficients of Eisenstein series are a sum of products of theta functions and Eisenstein series.
In Section 4, we will give a sketch of proof of the main theorem.
Except for $G=Sp_{4n}$, the situations are similar, in that we do not need to consider half-integral modular forms.
Finally  in Section 5, we discuss conjectures and problems related to the results in \cite{KY}.

\smallskip

\textbf{Acknowledgments.} We would like to thank H. Narita and S. Hayashida for their invitation to participate in the RIMS workshop on Modular Forms and Automorphic Representations on February 2-6, 2015.

\section{Description of tube domains}
\subsection{$Sp_{2n}$}
The tube domain is given by
$$\mathbb{H}_n:=\{Z\in M_{n}(\C)\ |\ {}^tZ=Z,\ {\rm Im}(Z)>0 \}\subset \C^{\frac{n(n+1)}{2}}$$
and $\gamma=\left(
\begin{array}{cc}
A & B \\
C & D
\end{array}\right)\in Sp_{2n}(\R)$ acts on $\mathbb{H}_n$ as
$\gamma(Z)=(AZ+B)(CZ+D)^{-1}$. Put $j(\gamma,Z)=\det(CZ+D)$.

\subsection{$SU_{2n+1}$}
The tube domain is given by
$$\mathcal{H}_{2n+1}:=\{Z\in M_{2n+1}(\C)\ |\ \frac{1}{2\sqrt{-1}}(Z-{}^t{\bar Z})>0 \}\subset \C^{(2n+1)^2}$$
and $\gamma=\left(
\begin{array}{cc}
A & B \\
C & D
\end{array}\right)\in SU_{2n+1}(\R)$ acts on $\mathcal{H}_{2n+1}$ as
$\gamma(Z)=(AZ+B)(CZ+D)^{-1}$. Put $j(\gamma,Z)=\det(CZ+D)$.

\subsection{$SU(2n,H)$}
Let $H$ be a definite quaternion algebra with basis ${1,i,j,k=ij}$ over $\Q$.
By Lemma 1.1 of \cite{Yamana}, there exists a unique polynomial map (with $4n$ variables) $P:M_n(H)\lra \Q$ such that
$\nu(X)=P(X)^2$ and $P(I_n)=1$. Put ${\rm Paf}(X)=P(X)$ for any $X\in M_n(H)$.
The tube domain is given by
$$\mathfrak{H}_{n}:=\{Z\in M_{n}(H\otimes_\Q\C)\ |\ {}^\ast Z=Z,\ {\rm Im}(Z)>0 \}\subset \C^{2n(n+1)}$$
and $\gamma=\left(
\begin{array}{cc}
A & B \\
C & D
\end{array}\right)\in SU(2n,H)(\R)$ acts on $\mathfrak{H}_n$ as
$\gamma(Z)=(AZ+B)(CZ+D)^{-1}$. Put $j(\gamma,Z)=\nu(CZ+D)^{\frac{1}{2}}$.

\subsection{$E_{7,3}$} This group is defined by using Cayley numbers and the structure is
rather complicated than previous cases.  We refer \cite{B},\cite{Coxeter}, \cite{kim}, and
Section 2 of \cite{KY}.
For any field $K$ whose characteristic is different from $2$ and $3$,
the Cayley numbers $\frak C_K$ over $K$ is an eight-dimensional vector space over $K$ with basis
$\{e_0=1,e_1,e_2,e_3,e_4,e_5,e_6,e_7\}$ satisfying certain rules of multiplication.
Let $\frak J_K$ be the exceptional Jordan algebra consisting of the element:
\begin{equation*}\label{x}
X=(x_{ij})_{1\le i,j\le 3}=\left(\begin{array}{ccc}
a& x & y \\
\bar{x} & b & z \\
\bar{y}& \bar{z} & c
\end{array}\right),
\end{equation*}
where $a,b,c\in Ke_0=K$ and $x,y,z\in \frak C_K$.
We also define
$$\frak J_2=\Bigg\{\left(\begin{array}{cc}
a& x  \\
\bar{x} & b
\end{array}\right) \Bigg|\ a,b\in K, \ x\in \frak C_K  \Bigg\}.
$$
Then the exceptional domain is
$$\frak D:=\{Z=X+Y\sqrt{-1}\in \frak J_\C\ |\ X,Y\in \frak J_\R,\ Y>0\}$$
which is a complex analytic subspace of $\C^{27}$ . We also define
$$\frak D_2:=\{X+Y\sqrt{-1}\in \frak J_2(\C)\ |\ X,Y\in \frak J_2(\R),\ Y>0\}
$$
which is the tube domain of $Spin(2,10)$, i.e., $Spin(2,10)$ acts on $Z\in \frak D_2$.

\section{Jacobi groups and Jacobi forms}
In this section we review Jacobi groups and Jacobi forms with a matrix index.

\subsection{Jacobi groups}
We are concerned with the Jacobi group $J$ realized in $G$, which is a semi-direct product
$J \simeq V\rtimes H$ of a semisimple group $H$ and a Heisenberg group $V$ with a 2 step unipotency which has a form
$V=X\cdot Y\cdot Z$, where each factor is an additive group (scheme), ${\rm dim}(X)={\rm dim}(Y)$, and
the center of $V$ is $Z$. We further require that the action of $H$ on $Z$ is trivial.

In our case, $H=SL_2$ if $G\not= SU_{2n+1}$, and $H=SU_1$ if $G=SU_{2n+1}$.
If we write an element as $v=v(x,y,z)$, then by definition, an alternating form $\langle \ast,\ast \rangle$ is
furnished on
$X\oplus Y$ such that the multiplication of two elements in $V$ is given by
$$v(x,y,z)\cdot v(x',y',z')=v(x+x',y+y',z+z'+\frac{1}{2}\langle (x,y),(x',y') \rangle)
$$
and further $SL_2$ or $U_1$ acts on $V$ as
$$\gamma\cdot v(x,y,z)=v(ax+cy,bx+dy,z),\ \gamma=\left(
\begin{array}{cc}
a & b \\
c & d
\end{array}\right)\in SL_2\ {\rm or}\ U_1.$$
We shall give a table of the dimension ${\rm dim}(X)$ of $X$ as a vector scheme over $\Z$ which will be related to the difference of the weights between the original form and the lift.
\begin{table}[htbp]
\begin{center}
{\renewcommand\arraystretch{2}
\begin{tabular}{|c|c|c|c|c|}
\hline
$G$  &  $Sp_{4n}$  & $SU_{2n+1}$ & $SU(2n,H)$ & $E_{7,3}$    \\
\hline
${\rm dim}(X)$  & $2n-1$  &  $4n$ & $4(n-1)$  & $16$    \\
\hline
\end{tabular}}
\end{center}
\caption{}
\end{table}

The difference between $\ell(k)$ and $2k$ is given by $\frac{1}{2}{\rm dim}(X)$ except for $Sp_{4n}$.
For $Sp_{4n}$, we first obtain a cusp form of the half-integral weight $k+\frac 12$  via Shimura correspondence
$S_{2k}(SL_2(\Z))\simeq S_{k+\frac{1}{2}}(\Gamma_0(4))^+$  from the cusp form $f\in S_{2k}(SL_2(\Z))$. Then the difference should be
understood as
$\ell(k)-(k+\frac{1}{2})=n-\frac{1}{2}$, which is nothing but $\frac{1}{2}{\rm dim}(X)$ for $Sp_{4n}$.

For $Sp_{4n}$,
$$V=\Bigg\{v(x,y,z)=\left(
\begin{array}{cccc}
1_{2n-1} & x & z & y \\
0 & 1 & {}^t y & 0 \\
  &   &     1_{2n-1}  & 0 \\
  &   &     -{}^t x & 1
 \end{array}\right)\in Sp_{4n} \Bigg|\ {}^tz-y({}^t x)=z-x({}^t y)    \Bigg\}=X\cdot Y\cdot Z,
$$
where
$X=\{v(x,0,0)\in V\},\ Y=\{v(x,0,0)\in V\}$, and $Z=\{v(0,0,z)\in V\}$, and
\begin{equation}\label{SL}
SL_2\simeq H:=\Bigg\{\left(
\begin{array}{cccc}
1_{2n-1} & 0 & 0_{2n-1} & 0 \\
0 & a &  0 & b \\
0_{2n-1}  & 0  &     1_{2n-1}  & 0 \\
 0 & c  &     0 & d
 \end{array}\right)\in  Sp_{4n} \Bigg\}.
\end{equation}

For $SU_{2n+1}$,
$$V=\Bigg\{v(x,y,z)=\left(
\begin{array}{cccc}
1_{2n} & x & z & y \\
0 & 1 & {}^t {\bar y} & 0 \\
  &   &     1_{2n}  & 0 \\
  &   &     -{}^t {\bar x} & 1
 \end{array}\right)\in SU_{2n+1} \Bigg|\ {}^t {\bar z}-y({}^t {\bar x})=z-x({}^t {\bar y})    \Bigg\}=X\cdot Y\cdot Z,
$$
where
$X=\{v(x,0,0)\in V\},\ Y=\{v(x,0,0)\in V\}$, and $Z=\{v(0,0,z)\in V\}$, and
\begin{equation}\label{SL-1}
U_1\simeq H:=\Bigg\{\left(
\begin{array}{cccc}
1_{2n} & 0 & 0_{2n} & 0 \\
0 & a &  0 & b \\
0_{2n}  & 0  &     1_{2n}  & 0 \\
 0 & c  &     0 & d
 \end{array}\right)\in SU_{2n+1} \Bigg\}.
\end{equation}

We omit details for $SU(2n,H)$ or $E_{7,3}$. Instead we refer Section 5 of \cite{Yamana}, and Section 3 and 4 of \cite{KY}.

Recall $L$ in the introduction. This is the parameter space of Fourier expansion of a modular form on $\D$.
Let $Z'$ be a subgroup of the unipotent radical $N$ of the Siegel parabolic subgroup
consisting of matrices whose last low and last column are zero. Then $Z'$ is naturally identified with $Z$.
We denote by $L'$ (resp. $L'_+$) the subset of $Z'(\Q)$ consisting of semi-positive (resp. positive), semi-integral matrices.
For any $T\in L^+$, there exists $S\in L'_+$ such that
$T=\left(
\begin{array}{cc}
S & \alpha \\
\beta  & x
\end{array}\right)$ with $x\in \Bbb Z_+$ and
$\beta=\begin{cases}  {}^t\alpha, &\text{if $G=Sp_{2n}$}\\
 {}^t{\bar \alpha}, &\text{if $G=SU_{2n+1}$ or $SU(2n,H)$}\\
{}^t {\bar\alpha}, &\text{if $G=E_{7,3}$}.\end{cases}
$

Henceforth we fix $S\in L'_+$. We define the map $\lambda_S$ on $Z$ by $z\mapsto \frac{1}{2}\tr_G(Sz)$ if $G\not=E_{7,3} $ and
$z\mapsto \frac{1}{2}(S,z)$ for $E_{7,3}$.
Then for any domain ring $R$ with characteristic zero, the map $V(R)\lra X\oplus Y\oplus R,\ v(x,y,z)\mapsto (x,y,\lambda_S(z))$
gives rise to the Heisenberg structure on $X\oplus Y\oplus R$. Hence
for any two elements $(x,y,a),(x',y',b)\in X\oplus Y\oplus R$, the multiplication is given by
$$(x,y,a)\ast(x',y',b)=(x+x',y+y',a+b+\frac{1}{2}\langle (x,y),(x',y') \rangle_S)$$
where $\langle (x,y),(x',y') \rangle_S=\sigma_S(x,y')-\sigma_S(x',y)$.
Here $\sigma_S(\ast,\ast)$ on $X\oplus Y$ is given by
   $$\sigma_S(x,y)= \left\{
\begin{array}{ll}
{}^t xS y  & {\rm if}\ G=Sp_{4n},\\
{}^t {\bar x}S y  & {\rm if}\ G=SU_{2n+1}\ {\rm or}\ SU(2n,H) \\
 (S,x({}^t\overline{y})+y({}^t\overline{x})) & {\rm if}\ G=E_{7,3}
\end{array}\right.
$$
Put $\mathfrak{X}:=X(\R)\otimes_\R\C$ and $\D_1:=\mathbb{H}\times \mathfrak{X}$.
The group $J(\R)$ acts on $\mathfrak{D}_1$ by
$$\beta(\tau,u):=\left(\gamma \tau,\frac{u}{c\tau+d}+x(\gamma \tau)+y \right),\
\beta=v(x,y,z)h,\ v(x,y,z)\in V(\R),\ h=h(\gamma)\in H(\R)$$ where
$\gamma=\left(\begin{array}{cc}
a & b \\
c & d
\end{array}\right)\in {\rm SL}_2(\R)$.
Here $\gamma \tau=\ds\frac{a\tau+b}{c\tau+d}$ and put $j(\gamma,\tau):=c\tau+d$ for simplicity.
For each positive half integer $k$, the automorphy factor on $J(\R)\times \mathfrak{D}_1$ is defined by
$$j_{k,S}(\beta,(\tau,u)):=j(\gamma,\tau)^{k}
\bold{e}(-2\lambda_S(z)+\frac{c}{j(\gamma,\tau)}\sigma_S(u,u)-\frac{2\sigma_S(x,u)}{j(\gamma,u)}-
\sigma_S(x,x)(\gamma\tau)-\sigma_S(x,y)),
$$
where $\bold{e}(\ast)=\exp(2\pi\sqrt{-1}\ast)$.
When $k$ is not an integer, $j(\gamma,z)^k=(j(\gamma,z)^\frac{1}{2})^{2k}$ is defined by the automorphy factor
$j(\gamma,z)^\frac{1}{2}$ of the metaplectic covering $\widetilde{SL}_2(\R)$
of $SL_2(\R)$.

For each function $f:\mathfrak{D}_1\lra \C$ and $\beta\in V(\R)$,
we define the ``slash" operator $f|_{k,S}[\beta]:\mathfrak{D}_1\lra \C$ by
$$f|_{k,S}[\beta](\tau,u):=j_{k,S}(\beta,(\tau,u))^{-1}f(\beta(\tau,u)).
$$

\subsection{Jacobi forms with a matrix index}
We define and study Jacobi forms of matrix index on $\mathfrak{D}_1=\mathbb{H}\times \mathfrak{X}$ in the classical setting.
Set
$$\Gamma_J:=J(\Q)\cap G(\Z).
$$

\begin{defin}\label{jacobi} Let $k$ be a positive even integer if $G\not=Sp_{4n}$, a
positive half-integral integer if $G=Sp_{4n}$, and $S$ be an element of $L'_+$.
We say a holomorphic function $\phi:\mathfrak{D}_1\lra \C$ is a Jacobi form (resp. Jacobi cusp form) of weight $k$ and index $S$ if
$\phi$ satisfies the following conditions:

\begin{enumerate}
\item $\phi|_{k,S}[\beta]=\phi$ for any $\beta\in \Gamma_J$
\item $\phi$ has a Fourier expansion of the form
$$\phi(\tau,u)=\sum_{\xi\in X(\Q),\ N\in \Z}c(N,\xi)\bold{e}(N\tau+\sigma_S(\xi,u)),
$$
where $c(N,\xi)=0$ unless $S_{\xi,N}:=
\left(\begin{array}{cc}
S & S\xi \\
{}^\star \xi S & N
\end{array}\right)$ belongs to $L'$ (resp. $L'_{+}$) where ${}^\star \xi$ stands for ${}^t \xi$ if $G=Sp_{4n}$,
${}^t {\bar\xi}$ if $G=SU_{2n+1}$ or $SU(2n,H)$, and ${}^t \bar{\xi}$ if $G=E_{7,3}$.
\end{enumerate}
We denote by $J_{k,S}(\Gamma_J)$ (resp. $J^{{\rm cusp}}_{k,S}(\Gamma_J)$) the space of Jacobi forms
(resp. Jacobi cusp forms) of weight $k$ and index $S$.
\end{defin}
Let us extend the quadratic form $\sigma_S$ linearly to that on
$\mathfrak{X}$.
We denote by $ \mathcal{S}(X(\A_f))$ the space of Schwartz functions on $X(\A_f)$.
For each $\varphi \in \mathcal{S}(X(\A_f))$,
the classical theta function on $\mathfrak{D}_1:=\mathbb{H}\times \mathfrak{X}$ is given by
$$\theta^S_\varphi(\tau,u):=\ds\sum_{\xi\in X(\Q)}\varphi(\xi)\bold{e}(\sigma_S(\xi,\xi)\tau+2\sigma_S(\xi,u)).$$

Define the dual of the lattice $\Lambda:=X(\Z)$ with respect to the quadratic form $\sigma_S$ by
$$\widetilde{\Lambda}(S)=\{x\in X(\Q)\ |\ \sigma_S(x,y)\in \Z {\rm \ for\ all}\ y\in \Lambda \}.$$
If $S\in L'_{+}$, then the quotient $\widetilde{\Lambda}(S)/\Lambda$ is a finite group.
Fix a complete representative $\Xi(S)$ of $\widetilde{\Lambda}(S)/\Lambda$ and denote by
$\varphi_\xi$ the characteristic function $\xi+\ds\prod_{p<\infty}X(\Z_p)\in \mathcal{S}(X(\A_f))$.
Any Jacobi form turns to be the linear combination of products of elliptic modular forms and
theta functions by the following lemma.
\begin{lemma}\label{product} Assume $S\in L'_{+}$.
Let $\Xi(S)$ be a complete representative of $\widetilde{\Lambda}(S)/\Lambda$. Then
any Jacobi form $\phi\in J_{k,S}(\Gamma_J)$ can be written as
$$\phi(\tau,u)=\sum_{\xi\in \Xi(S)}\phi_{S,\xi}(\tau) \theta^S_{\varphi_\xi}(\tau,u),\quad
\phi_{S,\xi}(\tau)=\sum_{N\in \Z \atop N-\sigma_S(\xi,\xi)\ge 0}c(N,\xi)\bold{e}((N-\sigma_S(\xi,\xi))\tau).
$$
Furthermore, for each $\xi\in \Xi(S)$,  $\phi_{S,\xi}(\tau)$ is an elliptic modular form of weight $k-\frac{1}{2}{\rm dim}(X)$.
\end{lemma}
\begin{proof} See example (iv) at Section 2 of \cite{Krieg} and also the argument at p.656 of \cite{Ik1}.
\end{proof}
Let $k$ be a positive integer and $F$ be a modular form of weight $k$ on $\mathfrak{D}$.
We rewrite the variable $Z$ on $\D$ as
$\left(\begin{array}{cc}
W & u \\
v & \tau
\end{array}\right)$ where $\tau\in\mathbb{H}$, $u\in X(\R)\otimes_\R\C$, and $W\in \mathbb{H}_{2n-1},\mathcal{H}_{2n},\mathfrak{H}_{n-1}$,or
$\mathfrak{D}_2$. Note that $v$ is determined by $u$.
Then we have the Fourier-Jacobi expansion
\begin{equation}\label{FJ}
F\left(\begin{array}{cc}
W & u \\
v & \tau
\end{array}\right)=\sum_{S\in  L'} F_S(\tau,u)\bold{e}((S,W)).
\end{equation}
\begin{lemma}\label{FJC} Keep the notation as above. Assume $S\in L'_+$.
Then $F_S(\tau,u)\in J_{k,S}(\Gamma_J)$.
\end{lemma}
\begin{remark}\label{important-rmk}
Consider any holomorphic function $F(Z)$ with $Z=\left(\begin{array}{cc}
W & u \\
v & \tau
\end{array}\right)$ on $\D$ which is invariant under $P(\Z)$.
Then one has
the Fourier and the Fourier-Jacobi expansion
$$F(Z)=\sum_{T\in  L}A_F(T)\bold{e}((T,Z))=\sum_{S\in  L'}F_S(\tau,u)\bold{e}((S,W)),
$$
as in (\ref{FJ}).
By Lemma \ref{product},
$$F_S(\tau,u)=\sum_{\xi\in \Xi(S)}F_{S,\xi}(\tau) \theta^S_{\varphi_\xi}(\tau,u),\quad
F_{S,\xi}(\tau)=\sum_{N\in \Z \atop N-\sigma_S(\xi,\xi)\ge 0}A_F(S_{\xi,N})\bold{e}((N-\sigma_S(\xi,\xi))\tau),
$$
where $S_{\xi,N}=\left(\begin{array}{cc}
S & S\xi \\
{}^\star \xi S & N
\end{array}\right)$. The function $F_{S,\xi}$ will be called $(S,\xi)$-component of $F$.
\end{remark}

Recall the following definition from \cite{Ik1, Ik2} .
\begin{defin}\label{family} For a sufficiently large $k_0$, a compatible family of Eisenstein series is a family of elliptic modular forms, for even
integer $k'\geq k_0$,
$$g_{k'}(\tau)=b_{k'}(0)+\sum_{N\in\Q_{>0}} N^{\frac{k'-1}{2}} b_{k'}(N) q^N,\ q=\bold{e}(\tau),
$$
satisfying the following three conditions:
\begin{enumerate}
\item $g_{k'}\in\mathcal V(E^1_{k'})$ for all $k'\geq k_0$
\item for each $N\in\Q_+^\times$, there exists $\Phi_N\in\mathcal R$ such that $b_{k'}(N)=\Phi_N(\{p^{\frac{k'-1}{2}}\}_p)$.
\item there exists a congruence subgroup $\Gamma\subset SL_2(\Z)$ such that $g_{k'}\in M_{k'}(\Gamma)$ for all $k'\geq k_0$.
Here $ M_{k'}(\Gamma)$ stands for the space of elliptic modular forms of weight $k$ with respect to $\Gamma$.
\end{enumerate}
\end{defin}

The following theorem plays a key role  in the proof of Theorem \ref{main-thm1}:

\begin{theorem}\label{eisen}
Keep the notations above.
Let $E_{\ell(k)}$ be the Siegel Eisenstein series in Section 1. Assume $S\in L'_+$.
Then any $(S,\xi)$-component of $E_{\ell(k)}$ is an Eisenstein series of weight $k-\frac{1}{2}{\rm dim}(X)$.
\end{theorem}
This theorem was first proved by B\"ocherer \cite{bocherer} for $G=Sp_{4n}$ in the classical language.
However the proof there involves many complicated terms and seems difficult to read off what we need.
More sophisticated proof was given by Ikeda \cite{Ik3}. He made a good use of Weil representation and
worked over the adelic language. In \cite{Yamana}, \cite{KY}, the authors followed his method.
However in case $E_{7,3}$, the group structure is much more complicated than others. So the proof
is not a routine at all.

The following Lemma 10.2 of \cite{Ik2} is a crucial ingredient.
\begin{theorem}\label{ikeda-lem} Let $f(\tau)=\ds\sum_{n=1}^\infty c(n)q^n$ be a Hecke eigenform of weight $k$ with respect to $SL_2(\Z)$ with
$c(p)=p^{\frac{k-1}{2}}(\alpha_p+\alpha^{-1}_p)$.
Assume that there is a finite dimensional representation $(u,\C^d)$ of $SL_2(\Z)$ and
$$\vec{\Phi}_N:={}^t(\Phi_{1,N},\ldots,\Phi_{d,N})\in\mathcal R^d,\ N\in\Q_{>0}$$ satisfying the following two conditions:
\begin{enumerate}
\item there exists a vector valued modular form $\vec{g}_{k'}={}^t(g_{1,k'},\ldots,g_{d,k'})$ which has
$$\vec{g}_{k'}(\tau)=\vec{b}_{k'}(0)+\sum_{N\in\Q_{>0}} N^{\frac{k'-1}{2}} \vec{b}_{k'}(N) q^n,\  (\vec{b}_{k'}(N)=
{}^t(b_{1,k'}(N),\ldots,b_{d,k'}(N)),\ N\in \Q_{\ge 0})$$ of weight $k'$ with type $u$
for each sufficiently large even integers $k'$, hence this means that
$$\vec{g}_{k'}(\tau)|_{k'}[\gamma]:={}^t(g_{1,k'}|_{k'}[\gamma],\ldots,g_{d,k'}|_{k'}[\gamma])=
u(\gamma)\vec{g}_{k}(\tau)\ {\rm for\ any}\ \gamma\in SL_2(\Z),$$
\item each component $g_{i,k'},\ (1\le i\le d)$ of $\vec{g}_{k'}(\tau)$ is a compatible family of Eisenstein series such that $$b_{i,k'}(N)=\Phi_{i,N}(\{p^{\frac{k'-1}{2}}\}_p).$$
\end{enumerate}
Then $\vec{h}(\tau):=\ds\sum_{N\in\Q_{>0}} N^{\frac{k-1}{2}} \vec{\Phi}_{N}(\{\alpha_p\}_p) q^N$ is
a vector valued modular form of weight $k$ with type $u$, hence it satisfies
$$\vec{h}(\tau)|_{k}[\gamma]=u(\gamma)\vec{h}\ {\rm for\ any}\ \gamma\in SL_2(\Z).$$
\end{theorem}

\section{Proof of Theorem \ref{main-thm1} and \ref{main-thm2}}
Recall that for each normalized Hecke eigenform $f=\ds\sum_{n=1}^\infty a(n)q^n\in S_{2k}(SL_2(\Z))$, we have considered the following formal series on $\D$:
$$F_f(Z):=\sum_{T\in L^+} A_{F_f}(T)\exp(2\pi\sqrt{-1}\tr_G(TZ)),\ Z\in \frak D, \quad
A_{F_f}(T)=A(\{\alpha_p\}_p).
$$
The first task is to check the absolute convergence for $F_f$:
This is done by using explicit formula of Fourier coefficients of Siegel Eisenstein series and
Ramanujan bound $|a(p)|\le 2p^{k-\frac{1}{2}}$.

Next, we use the fact that $\Gamma=G(\Bbb Z)$ is generated by $P(\Bbb Z)$ and $H(\Bbb Z)$, where $H$ is in (\ref{SL}) or (\ref{SL-1}).
We can easily check, by property of Fourier coefficients of Siegel Eisenstein series, that
$F_f$ is invariant under the action of $P(\Bbb Z)$.
Therefore to prove the automorphy of $F_f$, we have to check only the invariance of $F_f$ under the action of $H(\Bbb Z)$. For this, we need to use the Fourier-Jacobi expansion.

To unify notation we write the Fourier coefficient of $F_f$ as
$A_{F_f}(T)=C_1(T)C_2(T)^{k-\frac{1}{2}}\prod_{p}\widetilde{F}_{p}(T;\alpha_p)$ for $T\in L^+$ where $C_1(T)=L(1-k,\chi_T)$ if $G=Sp_{4n}$,
$C_1(T)=1$ otherwise, and other terms should be clear from the definition as in the introduction.
 Since $F(Z):=F_f(Z)$ is invariant under $P(\Bbb Z)$, by
Remark \ref{important-rmk}, one has the Fourier-Jacobi expansion:
\begin{equation*}\label{FJ1}
F\left(\begin{array}{cc}
W & u \\
v & \tau
\end{array}\right)=\sum_{S\in  L'_+}F_S(\tau,u)\bold{e}(\tr_G(SW)), \quad
F_S(\tau,u)=\sum_{\xi\in \Xi(S)}F_{S,\xi}(\tau) \theta^S_{\varphi_\xi}(\tau,u),
\end{equation*}
and
\begin{equation*}\label{FJ3}
\begin{array}{ll}
F_{S,\xi}(\tau)&=\ds\sum_{N\in \Z \atop N-\sigma_S(\xi,\xi)\ge 0}
A_F(S_{\xi,N})\bold{e}((N-\sigma_S(\xi,\xi))\tau),\ S_{\xi,N}:=\left(\begin{array}{cc}
S & S\xi \\
\eta S & N
\end{array}\right) \\
&=\ds\sum_{N\in \Z \atop N-\sigma_S(\xi,\xi)\ge 0}C_1(S_{\xi,N})C_2(S_{\xi,N})^{k-\frac{1}{2}}
\prod_p \widetilde{F}_{p}(S_{\xi,N};\alpha_p)\bold{e}((N-\sigma_S(\xi,\xi))\tau)\\
&=D(S)^{k-\frac{1}{2}}\ds\sum_{N\in \Z \atop N-\sigma_S(\xi,\xi)\ge 0}C_1(S_{\xi,N})(N-\sigma_S(\xi,\xi))^{k-\frac{1}{2}}
\prod_p\widetilde{F}_{p}(S_{\xi,N};\alpha_p)\bold{e}((N-\sigma_S(\xi,\xi))\tau)
\end{array}
\end{equation*}
where there exists the constant $D(S)$ depending only on $S$ such that $C_2(S_{\xi,N})=D(S)(N-\sigma_S(\xi,\xi))$.
The invariance under $H(\Z)$ is equivalent to claiming that $F_S(\tau,u)\in J_{k,S}(\Gamma_J)$ for any $S\in L'_+$.

By (2,1), p.124 of \cite{takase}, for each $\gamma\in SL_2(\Z)$, there exists a unitary matrix $u_S(\gamma)=
(u_S(\gamma)_{\xi \eta})_{\xi,\eta\in \Xi(S)}$ such that
\begin{equation*}\label{inv2}
\theta^S_{\varphi_\xi}|_{k,S}[\gamma](\tau,u)=\sum_{\eta\in \Xi(S)}u_S(\gamma)_{\xi\eta}\theta^S_{\varphi_\eta}(\tau,u).
\end{equation*}
Further there exists a positive integer $\Delta_S$ depending on $S$ such that $u_S$ is trivial on $\Gamma(\Delta_S)\subset SL_2(\Z)$.
Since $\{ \theta^S_{\varphi_\xi}\ |\  \xi\in \Xi(S) \}$ are linearly independent over $\C$,
it suffices to prove that $\{F_{S,\xi}\}_{\xi\in \Xi(S)}$ is a vector
valued modular form with type $u_S$.

For a sufficiently large positive integer $k'$,
we now turn to consider $(S,\xi)$-component $(\mathcal{E}_{\ell(k')})_{S,\xi}$ of the classical
Eisenstein series
$$ \mathcal{E}_{\ell(k')}(Z)=\sum_{T\in L}A(T)\exp(2\pi\sqrt{-1}\cdot \tr_G(TZ)),\
A(T)=C_1(T)C_2(T)^{k'-\frac{1}{2}}\prod_{p}\widetilde{F}_{p}(T;p^{k'-\frac{1}{2}}),$$
Then one has
$$
\begin{array}{l}
D(S)^{-k'+\frac{1}{2}}(\mathcal{E}_{\ell(k')})_{S,\xi}(\tau)\\
=
\ds\sum_{N\in \Z \atop N-\sigma_S(\xi,\xi)\ge 0}C_1(S_{\xi,N})(N-\sigma_S(\xi,\xi))^{-k'+\frac{1}{2}}
\prod_{p|\det(S_{\xi,N})}\widetilde{F}_{p}(S_{\xi,N};p^{k'-\frac{1}{2}})\bold{e}((N-\sigma_S(\xi,\xi))\tau)\\
\end{array}
$$
Then by Theorem \ref{eisen},
$\{D(S)^{-k'+\frac{1}{2}}(\mathcal{E}_{\ell(k')})_{S,\xi}\}_{k'\gg0}$ makes up a compatible family of Eisenstein series
in the sense of Ikeda (see Section 10 of \cite{Ik1} for $G=Sp_{4n}$ and Section 7 of \cite{Ik2} for other cases).
Applying Lemma \ref{ikeda-lem}, one can conclude that
$$F_{S,\xi}=D(S)^{k-\frac{1}{2}}\ds\sum_{n\in\Z_{>0}\atop
n=N-\sigma_S(\xi,\xi),\ N\in \Z}C_1(S_{\xi,N}) n^{k-\frac{1}{2}} \prod_{p|\det(S_{\xi,N})}\widetilde{F}_{p}(S_{\xi,N};p^{k'-\frac{1}{2}}) q^n,
$$
is a vector valued modular form with type $u_S$. The non-vanishing is easy to check except for $Sp_{4n}$.
In this case, a bit of careful study was needed (see p.651 of \cite{Ik1}). At any late one can prove the non-vanishing of $F_f$.

Since we know Satake parameters of $\pi_F$, it is easy to compute $L(s,\pi_F,St)$.
For $G=E_{7,3}$, we can use the Langlands-Shahidi method for the case $GE_7\subset E_8$ (cf. \cite{kim1}, section 2.7.8).

\section{Some conjectures and problems}
In this section we are concerning with some conjectures and problems related to the results in \cite{KY}.

\subsection{Conjectural Arthur parameter} It is worth considering the compatibility with Arthur conjecture in the case $E_{7,3}$:
We write the degree 56 standard $L$-function of $F:=F_f$ as
$$L(s,\pi_F,St)=L(s,Sym^3 \pi_f)\prod_{i=-4}^4 L(s+i,\pi_f) \prod_{i=-8}^8 L(s+i,\pi_f).
$$
This suggests the following parametrization of $\pi_F$:

Let $\mathcal L$ be the (hypothetical) Langlands group over $\Q$, and let $\rho_f : \mathcal L\lra SL_2(\C)$ be the 2-dimensional irreducible representation of $\mathcal L$ corresponding to $\pi_f$.
Let $Sym^n$ be the irreducible $(n+1)$-dimensional representation of $SL_2(\C)$. Note that if $n=2m-1$, $Im(Sym^n)\subset Sp_{2m}(\C)$, and if $n=2m$, $Im(Sym^n)\subset SO_{2m+1}(\C)$.  We have the tensor product maps $SL_2(\C)\times Sp_{2m}(\C)\lra Sp_{4m}(\C)$ and $SL_2(\C)\times SO_{2m+1}(\C)\lra Sp_{4m+2}(\C)$. Hence

$\rho_f\otimes Sym^{16}: \mathcal L\times SL_2(\Bbb C)\longrightarrow Sp_{34}(\Bbb C)$, and
$\rho_f\otimes Sym^{8}: \mathcal L\times SL_2(\Bbb C)\longrightarrow Sp_{18}(\Bbb C)$.

\noindent Let $Sym^3\rho_f: \mathcal L\times SL_2(\Bbb C)\longrightarrow Sp_4(\Bbb C)$ be the parameter of $Sym^3\pi_f$, where it is trivial on $SL_2(\Bbb C)$. Consider the parameter
$$
\rho=Sym^3\rho_f\oplus (\rho_f\otimes Sym^{16})\oplus (\rho_f\otimes Sym^{8}) : \mathcal L\times SL_2(\Bbb C)\longrightarrow
Sp_4(\Bbb C)\times Sp_{34}(\Bbb C)\times Sp_{18}(\Bbb C)\subset Sp_{56}(\Bbb C).
$$

Note that $E_7(\Bbb C)\subset Sp_{56}(\Bbb C)$. We expect that $\rho$ will factor through $E_7(\Bbb C)$, and give rise to a parameter
$\rho: \mathcal L\times SL_2(\Bbb C)\longrightarrow
E_7(\Bbb C)$, which parametrizes $\pi_F$.

\subsection{Ikeda lift as CAP form} If $G=Sp_{4n}$, the Ikeda lift $F_f$ is a CAP form. Namely, $\pi_F$ is nearly equivalent to the quotient of the induced representation
$$Ind_{P_{2,..,2}}^{Sp_{4n}} \ \pi_f|det|^{n-\frac 12}\otimes\pi_f|det|^{n-\frac 32}\otimes\cdots\otimes \pi_f|det|^{\frac 12},
$$
where $P_{2,...,2}$ is the standard parabolic subgroup of $Sp_{4n}$ with the Levi subgroup $GL_2\times\cdots\times GL_2$ 
($n$ factors) (see also p.114 of \cite{Ik0}).

If $G=E_{7,3}$, $\pi_F$ cannot be a CAP form in a usual sense since there are not many $\Q$-parabolic subgroups of $E_{7,3}$.
We expect that $\pi_F$ will be a CAP form in a more general sense: Namely, there exists a parabolic subgroup
$Q=M'N'$ of the split $E_7$, and a cuspidal representation $\tau=\otimes_p' \tau_p$ of $M'$, and a parameter $\Lambda_0$
such that
for all finite prime $p$, $\pi_p$ is a quotient of $Ind_{Q(\Q_p)}^{E_7(\Q_p)} \ \tau_p\otimes exp(\Lambda_0,H_Q(\ )).$

\subsection{Miyawaki type lift to $GSpin(2,10)$} This work is in progress \cite{KY1}.
For $Z\in \frak D_2$, let $\begin{pmatrix} Z&0\\0&\tau\end{pmatrix}\in \frak D$.
For $f\in S_{2k}(SL_2(\Z))$, let $F$ be the Ikeda lift of $f$, which is a cusp form of weight $2k+8$ on $\frak D$.
For $h\in S_{2k+8}(SL_2(\Z))$, consider the integral

$$\mathcal F_{f,h}(Z)=\int_{SL_2(\Z)\backslash \Bbb H} F\begin{pmatrix} Z&0\\0&\tau\end{pmatrix} \overline{h(\tau)} (Im \tau)^{2k+6}\, d\tau.
$$

When $\mathcal F_{f,h}$ is not zero, it is a cusp form of weight $2k+8$ on $\frak D_2$. It is expected that $\mathcal F_{f,h}$ is a Hecke eigen form, and it would give rise to a cuspidal representation $\pi_{\mathcal F_{f,h}}$ on $GSpin(2,10)$: Let 
$\pi_{\mathcal F_{f,h}}=\pi_\infty\otimes \otimes_p' \pi_p$.
Let $\{\alpha_p, \alpha_p^{-1}\}$ and $\{\beta_p, \beta_p^{-1}\}$ be the Satake parameter of $f$ and $h$ at the prime $p$, resp. Then for each prime $p$, it is expected that the Satake parameter of $\pi_p$ is
$$\{(\beta_p\alpha_p)^{\pm 1}, (\beta_p\alpha_p^{-1})^{\pm 1}, 1, 1, p^{\pm 1},  p^{\pm 2} , p^{\pm 3}\}.
$$

Then the standard $L$-function of $\pi_{\mathcal F_{f,h}}$ is
$$L(s, \pi_{\mathcal F_{f,h}}, St)=L(s,h\times f)\zeta(s)^2 \zeta(s\pm 1)\zeta(s\pm 2)\zeta(s\pm 3),
$$
where the first factor is the Rankin-Selberg $L$-function.
This can be explained by Arthur parameter as follows: Let $\phi_f, \phi_h: \mathcal L\lra SL_2(\C)$ be the hypothetical Langlands parameter. Then due to the tensor product map $SL_2(\C)\times SL_2(\C)\lra SO_4(\C)$, we have
$\phi_f\otimes\phi_h: \mathcal L\lra SO_4(\Bbb C).$ The distinguished unipotent orbit $(7,1)$ of $SO_8(\C)$ gives rise to the map 
$SL_2(\C)\lra SO_8(\C)$. It defines the map 
$\phi_u: \mathcal L\times SL_2(\Bbb C)\lra SO(8,\Bbb C)$. Then consider
$$\phi=(\phi_h\otimes \phi_f)\oplus  \phi_u : \mathcal L\times SL_2(\C)\lra SO_4(\C)\times SO_8(\C)\subset GSO_{12}(\C).
$$
We expect that $\phi$ parametrizes $\pi_{\mathcal F_{f,h}}$.

\subsection{Pertersson formula and its possible application}
In case $E_{7,3}$, it may be interesting to give an explicit formula of
the Petersson inner product formula for $F_f$. (See \cite{ichino-ikeda} for its importance.) Since $L(s,\pi_F,St)$ involves
the third symmetric power $L$-function $L(s, Sym^3 \pi_f)$, we expect to somehow figure out an ``algebraic part" of
$L(s,Sym^3 \pi_f)$.

\end{document}